\documentclass{article}[12pt]
\usepackage{amsmath,amsfonts,amsthm}
\usepackage{amssymb}

\usepackage{graphics}
\usepackage{graphicx}

\usepackage{latexsym}
\usepackage{graphics}
\usepackage{amsmath}
\usepackage{amsthm}
\usepackage{xspace}
\usepackage{amssymb}
\usepackage{color}
\usepackage{cite}
\usepackage{latexsym}
\usepackage{graphics}
\usepackage{amsmath}
\usepackage{amsthm}
\usepackage{xspace}
\usepackage{amssymb}
\usepackage{epsfig}

\newcommand{\beql}[1]{\begin{equation}\label{#1}}
\newcommand{\eeql}{\end{equation}}
\newcommand{\eqn}[1]{(\ref{#1})}

\newcommand{\R}{\mathbb{R}}
\newcommand{\pr}{\mathbb{P}}
\newcommand{\E}{\mathbb{E}}

\newcommand{\Z}{\mathbb{Z}}

\newcommand{\mb}[1]{\mbox{\boldmath $#1$}}

\newtheorem{thm}{Theorem}
\newtheorem{lem}[thm]{Lemma}
\newtheorem{prop}[thm]{Proposition}
\newtheorem{cor}[thm]{Corollary}
\newtheorem{assumption}[thm]{Assumption}

\begin{document}

\title{Multiclass multiserver queueing system in the Halfin-Whitt heavy traffic regime.
Asymptotics of the stationary distribution}

\author
{
David Gamarnik \\
Operations Research Center \\
and Sloan School of Management\\
100 Main Street\\
MIT, Cambridge, MA,  02139\\
\texttt{gamarnik@mit.edu}
\thanks{
Research supported in part by the NSF grants CMMI-0726733.}
\and
Alexander L. Stolyar \\
Bell Labs, Alcatel-Lucent\\
600 Mountain Ave., 2C-322\\
Murray Hill, NJ 07974 \\
\texttt{stolyar@research.bell-labs.com}
\thanks{The authors wish to thank the Newton Institute of Mathematics, Cambridge, UK, 
where this work was initiated and partly conducted, for their hospitality.}
}

\date{\today}

\maketitle

\begin{abstract}
We consider a heterogeneous queueing system consisting of one large pool of
$O(r)$
identical servers, where $r\rightarrow\infty$ is the scaling parameter.
The arriving
customers belong to one of several classes which determines the service times in the distributional
sense.
The system is heavily loaded in the Halfin-Whitt sense, namely the nominal
utilization is $1-a/\sqrt{r}$
where $a>0$ is the spare capacity parameter.
Our goal is to obtain bounds on the steady state performance metrics such as the number of customers waiting in the queue  $Q^r(\infty)$.
While there is a rich literature on deriving process level (transient) scaling limits for such
systems, the results for steady state are primarily limited to the single class case.

This paper is the first one to address the case of heterogeneity in the steady state regime.
Moreover, our results hold for any service policy which does not admit server idling
when there are customers waiting in the queue.
We assume that the interarrival and service times have exponential distribution,
and that customers of each class may abandon while waiting in
the queue at a certain rate (which may be zero).
We obtain upper bounds of the form $O(\sqrt{r})$ on both $Q^r(\infty)$
and the number of idle servers.
The bounds are uniform w.r.t. parameter $r$ and the service policy.
In particular, we show
that $\limsup_r \E \exp(\theta r^{-{1\over 2}}Q^r(\infty))<\infty$.
Therefore, the sequence $r^{-{1\over 2}}Q^r(\infty)$ is tight and has
a uniform exponential tail bound.
We further consider the system
with strictly positive abandonment rates,
and show that
in this case
every weak limit $\hat Q(\infty)$ of $r^{-{1\over 2}}Q^r(\infty)$
has a sub-Gaussian tail. Namely $\E[\exp(\theta (\hat Q(\infty))^2)]<\infty$,
for some $\theta>0$.
\end{abstract}

\section{Introduction}
\label{section:intro}
Recently we have witnessed an explosion of interest in large scale queueing systems operating
in the so-called Halfin-Whitt heavy traffic regime. Such systems are attractive for modeling
large scale service systems such as call/contact services
or computer farms (used in cloud computing), as they achieve high utilization
and low service delay at the same time~\cite{AksinArmonyMehrotra},\cite{BrownCallCenters},\cite{GansKooleMandelbaum}.
Most of the prior literature on the subject, including the paper by Halfin and Whitt~\cite{HalfinWhitt81}, which has
initiated the subject, focused on the single class case where customers with identical in distribution
service requirements are processed by a large pool of identical servers. This includes the earlier follow up papers such as
\cite{ReimanPuhalskii2000}, which focused on the case of exponential and phase-type distribution of model primitives,
as well as more recent works~\cite{Reed},\cite{ReimanPuhalskii2000},\cite{JelenkovicMandelbaumMomcilovic},\cite{MandelbaumMomcilovic},
\cite{GamarnikMomcilovicCallCenters},\cite{KangRamanan},\cite{KaspiRamanan},\cite{GamarnikGoldbergGGN}.

The multiclass models have been analyzed  as well~\cite{TezcanDaiOR},\cite{StolyarTezcan},\cite{AtaGurvich},\cite{AtarMandelbaumReiman},
\cite{AtarMandelbaumReimanAAP}, specifically in the context of finding asymptotically optimal control policies, but primarily in the transient as opposed to the steady state regime. In the multiclass environment it is assumed that the processing times may depend on the customer class. (Furthermore, in more general models with heterogenous servers,
the processing times may depend on customer class and server type combination.)
The lack of literature analyzing multiclass models in steady state is unfortunate, since  multiclass systems offer a richer set
of models capturing, for example, call centers with agent skill and customer requirement heterogeneity. The results obtained for
single class models for the steady state regimes, unfortunately, are not applicable in the multiclass setting. The original
paper by Halfin and Whitt~\cite{HalfinWhitt81} relies on the Erlang's formula for the $M/M/N$ queueing systems. \cite{ReimanPuhalskii2000}
only analyze the transient regime. The remaining aforementioned
papers~\cite{Reed},\cite{ReimanPuhalskii2000},\cite{JelenkovicMandelbaumMomcilovic},\cite{MandelbaumMomcilovic},
\cite{GamarnikMomcilovicCallCenters},\cite{KangRamanan},\cite{KaspiRamanan},\cite{GamarnikGoldbergGGN}
covering the steady state regime also use the single class assumption in an important way. Thus a new set of tools
need to be developed to address heterogenous models.

In this paper we consider a multiclass queueing system with $O(r)$ parallel
identical servers in
the Halfin-Whitt regime, where parameter $r\to\infty$ and the system (nominal) utilization
 scales as $1-O(1/\sqrt{r})$.
The multiplicity
of classes comes from heterogeneity of customer mean service times.
Namely, customers class $i$ have the mean service requirement $1/\mu_i$.
The arrival processes of different types are assumed to be independent Poisson
(of rate $O(r)$), and the service time distributions are assumed to be independent
exponential.
Class $i$ customers waiting in the queue abandon the system at rate $\nu_i\ge 0$,
also in the Markov fashion.
An arbitrary non-idling, state dependent (in a sense to be explained) customer service
(scheduling) policy is implemented.
Interestingly, the details of this policy are irrelevant for our results.
Moreover, it is allowed to have a different policy for different values of parameter $r$.
Our results establish, in particular, that
the steady state (random) queue length denoted by $Q^r(\infty)$ scales like $O(\sqrt{r})$
as $r\rightarrow\infty$.
To be precise, 
we show that for some fixed parameter $\theta>0$, 
$\limsup_{r\rightarrow\infty} \E\left[\exp\left(\theta r^{-{1\over 2}}Q^r(\infty)\right)\right]<\infty$, which in particular ensures that the sequence
of distribution of $(1/\sqrt{r})Q^r(\infty)$ is tight.
If in addition, the scheduling policy
admits a process level limit $\hat Q(t)$ of the processes $\hat Q^r(t)=(1/\sqrt{r})Q^r(t)$,
then our results imply that any weak limit
(along a subsequence) of
the sequence of stationary distributions of
$\hat Q^r$  is the stationary distribution of $\hat Q$;
in other words, we establish an interchange of heavy traffic and steady state limits.
Thus our paper continues an earlier stream of works establishing interchange of limits for queueing
networks~\cite{GamarnikZeevi},\cite{BudhirajaLee},\cite{GamarnikMomcilovicCallCenters},\cite{PieraMazumdarGuillemin},\cite{Katsuda}.

Additionally, for the case when all abandonment rates $\nu_i$ are strictly positive,
we provide a sub-Gaussian bound on the tail of $\hat Q(\infty)$, where $\hat Q(\infty)$
is any weak limit of $\hat Q^r(\infty)=(1/\sqrt{r})Q^r(\infty)$.
A lot of recent literature
focused on queueing models with abandonments~\cite{GarnettMandelbaumReiman},\cite{DaiHe},\cite{MandelbaumMomcilovicAbandonments}.
It was observed that at a process level the presence of abandonments
leads to a Ornstein-Uhlenbeck type behavior for the limiting process. Thus one would expect that the steady state measures (of the limit process)
should have a sub-Gaussian tail. This was explicitly conjectured in \cite{DaiHe}. In this
paper we qualitatively confirm the conjecture.
One might naively conjecture, that in parallel to our exponential-tail result
described earlier, one would have
$\limsup_r \E\left[\exp\left(\theta (\hat Q^r(\infty))^2\right)\right]<\infty$, for some small enough $\theta$.
But is does not hold even in the single-class model, and we provide a simple illustration why, later in the paper.
However, we do show that any weak limit $\hat Q(\infty)$ of the sequence
$\hat Q^r(\infty)$ satisfies $\E[\exp(\theta (\hat Q(\infty))^2)]<\infty$ for small enough $\theta>0$. Namely the weak
limits of the (scaled) queue length indeed have sub-Gaussian tails.

The remainder of the paper is organized as follows. The model is described and the main results are stated in the next section.
In Section~\ref{section:key bounds} we obtain some preliminary results.
Specifically, we obtain some useful bounds that follow from the lower
bounds derived by comparing with the $M/M/\infty$ model, and also
a monotonicity (with respect to abandonment rates) result.
Our first main result on the uniform exponential-tail bounds
is proved in Sections~\ref{section:mainresult1}
and \ref{section:mainresult13333}. Another main result --
the sub-Gaussian-tail bounds on the weak limits of stationary distributions,
in the case of non-zero abandonment rates --
is proved in Section~\ref{section:mainresult2}.

We close this section with some notational conventions. For any real number $b$, let $b^+=\max(b,0),b^-=-\min(b,0)$, $\lfloor b \rfloor$ be the largest integer not greater than $b$, and
$\lceil b \rceil$ be the smallest integer not less than $b$.
$\R,\R_+,\Z,\Z_+$ denote the set of real values, non-negative real values, integer values and non-negative integer values, respectively.
$\Rightarrow$ denotes convergence in distribution.
$I\lbrace\cdot\rbrace$ is the indicator function taking value $1$ when the event inside $\lbrace\cdot\rbrace$ takes place,
and $0$ otherwise.

\section{Model and results}
\label{section:model and results}

We consider a sequence of queueing systems, indexed by parameter $r=1,2,\ldots$,
increasing to infinity.
The system with index $r$ has $N^r=r+a\sqrt{r}$ identical severs, for some fixed $a>0$.
(To be precise, this quantity needs to be integer, so it should be,
for example, $\lceil r+a\sqrt{r}\rceil$.
This subtlety does not create any
difficulties, besides clogging notation. To improve exposition,
let us assume that $r+a\sqrt{r}$ itself is integer.) We assume that there
is a finite set $I$
of
Poisson input customer
flows $i$ called classes with rates $\lambda_i r, i\in I$. Each customer of class $i\in I$
has an exponentially distributed service time (requirement) with mean $1/\mu_i$,
where (the service rate) parameter $\mu_i>0$.
At any given time a customer is either in service (by one server) or is waiting in the
queue of unlimited capacity; each server can serve at most one customer at any given time.
If and when the cumulative amount of time the customer spends in service
(which can, in principle, be preempted and resumed multiple times)
reaches its service requirement, the customer leaves the system.
Upon arriving into the system, for each class $i$ customer a random patience
time is generated according to the exponential distribution with
mean $1/\nu_i$, where (the abandonment rate)
parameter $\nu_i\ge 0$. If $\nu_i= 0$, the patience time is infinite.
If and when the cumulative amount of time the customer spends in the queue
(a customer can, in principle, move between queue and service multiple times)
reaches its patience time, the customer abandons and leaves the system.
Speaking less formally, a class $i$ customer leaves the system
in a small time interval $dt$ with probability $\mu_i dt$ if it is in service,
and with probability $\nu_i dt$ if it is in the queue.

We assume that $\sum_i\lambda_i/\mu_i=1$.
Then the nominal utilization of the system is
\begin{align*}
\frac{\sum_i\lambda_i r/\mu_i}{r+a\sqrt{r}}=1-{a\over \sqrt{r}+a}=1-{a\over \sqrt{r}}-o\left({1\over \sqrt{r}}\right),
\end{align*}
namely the queueing system is in the Halfin-Whitt heavy traffic regime.
Notice that when all abandonment rates $\nu_i=0$, the nominal utilization is the actual average
utilization.
Introduce the following notations:
$\rho_i\triangleq\lambda_i/\mu_i$ (implying $\sum_i \rho_i =1$),
$\mu_{\min}\triangleq \min_i \mu_i$ and
$\mu_{\max}\triangleq \max_i \mu_i$, $\nu_{\min}\triangleq \min_i \nu_i$ and
$\nu_{\max}\triangleq \max_i \nu_i$.
For every time $t\in\R_+$, let $Z_i^r(t)$ be the number of type $i$ customers in the system
both in service and waiting in the queue, and
$\Psi_i^r(t)$ be the number of type $i$ customers in service at time $t$. Then $Q_i^r(t)=Z_i^r(t)-\Psi_i^r(t)$
is the number of type $i$ customers waiting in the queue at time $t$.

We now formally define a family of service disciplines under the consideration.
Rather than specifying this family directly we describe the properties we want this family to satisfy. Namely, we take a very general point of view that the system evolution
is described by some countable continuous time Markov chain,
which has a projection $\left((Z^r_i(t),i\in I),(\Psi^r_i(t),i\in I)\right)$,
and which satisfies Assumption~\ref{assumption:Basic} below.
We then demonstrate that under many
interesting policies, the corresponding Markov processes indeed satisfy this
assumption. 

All formal results and proofs in this paper are concerned with abstract
processes satisfying Assumption~\ref{assumption:Basic}. 
Thus, the physical notions of customer arrivals and departures,
introduced earlier in this section,
can be viewed as no more than ``interpretations'' 
of increments and decrements of components $Z_i$ of the process.
But, of course, it will be useful to keep such physical interpretations
in mind, as they provide important insights into our methods.

\begin{assumption}\label{assumption:Basic}
Assume that for each $r$ there exists a Markov process $(X^r(t),t\in\R_+)$
in a countable state space $\mathcal{X}^r$ with the following properties:
\begin{enumerate}

\item Vector $\left((Z^r_i(t),i\in I),(\Psi^r_i(t),i\in I)\right) \in \Z_+^{2|I|}$,
such that $\Psi^r_i(t)\le Z^r_i(t)$ 
for all $i$, is a deterministic
function $f$ of $X^r(t)$; this function is such that for every value
${\bf w}$ within its range, the
pre-image $f^{-1}({\bf w})$ is finite cardinality.
Vector $(Q^r_i(t),i\in I)\in \Z_+^{|I|}$
is defined by $Q^r_i(t)=Z^r_i(t)-\Psi^r_i(t)$, and therefore is a deterministic
function of $X^r(t)$ as well.\\
 {\em Interptetation:} The process
$(Z^r_i(t),\Psi^r_i(t), i\in I)$, which describes the performance of the system
we are interested in, is a projection of some ``more detailed''
underlying Markov process
$X^r(t)$. The number of states of $X^r(t)$ that correspond to any given value
of $(Z^r_i(t), \Psi^r_i(t), i\in I)$ is finite.

\item The following relation holds:
\begin{align}
\label{cond-non-idle}
\sum_{i\in I}\Psi_i^r(t)=\min\lbrace N^r,\sum_{i\in I}Z_i^r(t)\rbrace.
\end{align}
{\em Interptetation:} Whatever scheduling policy the process $X^r(t)$
``encodes'', it is non-idling.

\item The transition rates of the process $X^r(t)$ are such that, for any state $X^r(t)$:

\begin{enumerate}
\item
The total rate of all transitions
such that vector $(Z_i^r(t), i\in I)$ changes to $(Z_i^r(t), i\in I)+\mb{e}_{\ell}$
is $\lambda_{\ell} r$, where $\mb{e}_i$ is the $i$-th unit vector in $\R^{|I|}$.

\item
If $Z_{\ell}^r(t)\ge 1$, then
the total rate of all transitions
such that vector $(Z_i^r(t), i\in I)$ changes to $(Z_i^r(t), i\in I)-\mb{e}_{\ell}$
is $\mu_{\ell} \Psi_{\ell}^r(t) + \nu_{\ell}Q^r_{\ell}(t)$.

\item The transitions such that any component of  $(Z_i^r(t), i\in I)$
increases or decreases by more than $1$, as well as transitions changing
more than one component of $(Z_i^r(t), i\in I)$, have zero rates.

\end{enumerate}
{\em Interptetation:} Scheduling policy is consistent with the notions of 
input processes, service
times and patience described at the beginning of this section. 
Note that there might be
positive rate transitions that ``reshuffle'' the values of occupancies
$(\Psi^r_i(t), i\in I)$ in an arbitrary way, as long as
conditions \eqn{cond-non-idle} and
$\Psi^r_i(t)\le Z^r_i(t)$ are not violated.

\item Markov process $X^r(t)$ is irreducible. \\
{\em Interptetation:} This condition is for convenience only --
it allows us to claim not only existence, but also uniquenees of
stationary distribution (as discussed below). Also,
this condition is very non-restrictive -- typically
the ``encoding process'' $X^r(t)$ can be constructed in a way such that
there is only one state corresponding to the ``empty'' system (with all
$Z_i^r=0$), and this state obviously is reachable from any other.

\end{enumerate}
\end{assumption}

Assumption~\ref{assumption:Basic} encompass a broad variety of
non-idling disciplines, with or without service preemptions.
For illustrative purposes, let us discuss the Class Priority (CP)
and First-In-First-Out (FIFO) examples.
The CP
 policy is fixed  by setting a total order $\prec$ on classes which without the loss of generality
is assumed to be $1\prec 2\prec \cdots\prec |I|$. This means that the jobs in class $|I|$ have the highest priority, the jobs in class $|I|-1$ have the next highest, etc. At any time,
the servers give priority to the jobs with the highest priority in a non-preemptive
or preemptive fashion.
Namely, every server upon service completion inspects the queue and selects
for service a job with the highest priority level, ties broken arbitrarily.
For the non-preemptive version,
once the service is initiated it is completed without interruption; in this case,
we can simply take $X^r(t)=(Z_i^r(t),\Psi_i^r(t), i\in I)$.
For the premptive version, if an arriving customer of type $i$ finds all servers busy
and if $j<i$, where $j$ is the lowest priority being served,
the arriving customer will force one of the type $j$ customers out to the queue and
replace it in service; in this case, the process is even simpler,
$X^r(t)=(Z_i^r(t), i\in I)$, because $(\Psi_i^r(t), i\in I)$ is uniquely determined
by $(Z_i^r(t), i\in I)$. In either case,
the ``empty'' state, with all $Z_i^r(t) = \Psi_i^r(t)=0$, is reachable from any other,
and therefore the chain is irreducible.

For the FIFO policy,
the statespace $\mathcal{X}^r$ of chain $X^r(t)$
is the set of finite sequences
$x_1x_2\cdots x_m$, where $x_n\in I$.
Each sequence represents the set of customers in the system,
in the order of their arrivals. The first $\min\{m,N^r\}$ customers are being served,
and the remaining ones wait in queue. Clearly, $(Z_i^r(t),\Psi_i^r(t), i\in I)$
is uniquely determined by $X^r(t)$. The empty state is reachable from any other.

We will now discuss
several general properties of the process $X^r(t)$ that follow directly
from Assumption~\ref{assumption:Basic}. For $\mb{x}\in \mathcal{X}^r$,
denote by $T(\mb{x})$  the set of states to which
a direct transition from $\mb{x}$ is possible,
and by $T^{-1}(\mb{x})$ the set of states from which
a direct transition to $\mb{x}$ is possible.
It is easy to observe that for any $\mb{x}$
both $T(\mb{x})$ and $T^{-1}(\mb{x})$ are finite cardinality -- this
follows from the fact that there is only a finite number of states
corresponding to each $\left((Z^r_i(t),i\in I),(\Psi^r_i(t),i\in I)\right)$,
and then to each $(Z^r_i(t),i\in I)$, and direct transitions
cannot cause any $Z^r_i(t)$ to change by more than $1$.
We can also observe that, starting any initial state,
with probability $1$ the process makes at most a finite number of
transitions in any finite time interval; indeed, 
with probability $1$ each $Z^r_i(t)$
remains finite on finite intervals, because the total rate of all transitions
causing it to increase (by $1$) is constant and equal $\lambda_i r$,
and then the number of states visited on finite intervals is finite,
from which we easily obtain that the number of transitions is finite as well.
This means that the process satisfying Assumption~\ref{assumption:Basic}
is well defined by its transition rates -- 
with probability $1$ there are no ``explosions''.

Since the system nominal utilization is strictly less than $1$
and condition \eqn{cond-non-idle} must hold,
the total workload (expected unfinished work) $\sum_i Z^r_i(t)/\mu_i$
must have a negative average drift, as long as $\sum_i Z^r_i(t)/\mu_i\ge N^r/\mu_{min}$.
(The latter condition implies $\sum_i Z^r_i(t)\ge N^r$, in which case
the average rate at which unfinished work in the system is processed is strictly
greater than the average rate at which new work arrives.)
This, along with the
fact that for any $K\ge 0$ there is only a finite number of states
with $\sum_i Z^r_i(t) \le K$, and the
assumption that
Markov chain $X^r(t)$ is irreducible, implies (using standard methods)
that {\em $X^r(t)$ is positive recurrent with a unique invariant probability
distribution $\pi^r$.}
From now
on we write $X^r(\infty)$ for the steady state version of $X^r(t)$.
The same applies to the variables determined by $X^r(t)$.
For example $Q^r(\infty)$ denotes the steady state queue size.

We remark that the irreducibility assumption is made to simplify the
exposition -- it is
non-essential, in the following sence.
If we do not require irreducibility, it is easy to show (again, using standard
methods) that Markov chain $X^r(t)$ has at most a finite number of
classes of communicating states, reachable w.p.1 from any transient state.
Restricted to any of those state classes, $X^r(t)$ is positive recurrent,
with unique stationary distribution.
Thus, stationary distributions of $X^r(t)$ exist and are characterized as 
convex combinations of a fixed finite set of distributions.
All results of this paper would still hold if instead of
{\em the} stationary distribution we would consider {\em a} stationary distribution.

Next we discuss some properties related to the (infinitesimal) generator $A^r$
of the Markov process $X^r(t)$.
The domain of generator $A^r$ might not contain all functions $F:\mathcal{X}^r\rightarrow \R$.
However, the following operator $\bar{A}^r$  is defined for {\em any} function $F$:
$$
\bar{A}^r F(\mb{x}) = \sum_{\mb{y}\in T(\mb{x})} \xi(\mb{x},\mb{y})(F(\mb{y})-F(\mb{x})),
~~ \mb{x}\in \mathcal{X}^r,
$$
where $\xi(\mb{x},\mb{y})$ is the transition rate from $\mb{x}$ to $\mb{y}$.
Given properties of our process $X^r(t)$, the following property is 
easily verified
directly: any function $F$, which is constant everywhere
except a finite subset of $\mathcal{X}^r$, is within the domain of
$A^r$ and, moreover,
$$
A^r F(\mb{x}) = \bar{A}^r F(\mb{x}), ~~\forall \mb{x},
$$
\begin{align}
\label{eq:ExpectedGeneratorZero}
\E A^r F(X^r(\infty))=\E \bar{A}^r F(X^r(\infty))=0.
\end{align}

Let us additionally denote
 $\Phi_i^r(t)=Z_i^r(t)/\mu_i$. Namely, $\Phi_i^r(t)$ is the total workload
at time $t$ associated with class $i$ jobs
- the sum of expected remaining service times of all class $i$ jobs present in the system at time $t$. The total workload is $\Phi^r(t)= \sum_i \Phi_i^r(t)$.

Now let us introduce the diffusion scaled version of our processes:
\begin{align*}
\hat Z_i^r(t)={Z_i^r(t) -\rho_i r\over \sqrt{r}},
\qquad \hat \Psi_i^r(t)={\Psi_i^r(t) -\rho_i r\over \sqrt{r}},
\qquad \hat \Phi_i^r(t)={\Phi_i^r(t) -\rho_i r/\mu_i \over \sqrt{r}}
={Z_i^r(t) \over \mu_i}
\qquad \hat Q_i^r(t)={Q_i^r(t)\over \sqrt{r}}.
\end{align*}
Furthermore, let $\hat Z^r(t)=\sum_i\hat Z_i^r(t)$, $\hat \Phi^r(t)=\sum_i\hat \Phi_i^r(t)$,
$\hat Q^r(t)=\sum_i\hat Q_i^r(t)$.
Namely, $\hat Z^r(t)$ is the diffusion scaled total number of customers
in the system,
$\hat \Phi^r(t)$ is the diffusion
scaled total workload in the system,
and $\hat Q^r(t)$ is the diffusion scaled total queue length.
Our method of proof is based on the Lyapunov function technique.
We will be considering
Lyapunov functions of the form $\exp(\theta \hat \Phi^r)$
and $\exp(\theta (\hat \Phi^r)^2)$.
The random variables $Z^r(t),\hat Z^r(t), \Phi_i^r(t)$, etc., corresponding to the steady state regime are denoted using the same notation,
except $\infty$ replaces $t$. For example $\hat Z^r(\infty)$ is the steady-state number customers in the system after the diffusion rescaling.
In addition, we  write $\hat Z^{r,+}(\infty)$
to mean $(\hat Z^r(\infty))^+$, $\hat Z_i^{r,+}(\infty)$
to mean $(\hat Z_i^r(\infty))^+$, etc.

We now ready to state our main results.
\begin{thm}
\label{th-main-no-abandonments}
The following properties hold.\\
(i) If $\nu_i \le \mu_i$ for all $i$ (including possibly
$\nu_i=0$ for some $i$), then for every $\theta>0$
\begin{align}
\label{eq-est3331}
\limsup_{r\rightarrow\infty} \E\exp\left(\theta \sum_i \hat Z_i^{r,-}(\infty)\right)<\infty.
\end{align}
(ii) There exists $\bar\theta>0$ such that for
every $\theta\in [0,\bar\theta]$
\begin{align}
\label{eq-est3332}
\limsup_{r\rightarrow\infty} \E\exp\left(\theta \sum_i \hat Z_i^{r,+}(\infty)\right)<\infty.
\end{align}
(iii) There exists $\bar\theta'>0$ such that for
every $\theta\in [0,\bar\theta']$
\begin{align}
\label{eq-est3333}
\limsup_{r\rightarrow\infty} \E\exp\left(\theta \sum_i \hat Z_i^{r,-}(\infty)\right)<\infty.
\end{align}
(iv) In particular, 
the sequence $\hat Z^r(\infty), r\in\Z_+$ is tight.
\end{thm}
Our proof approach can lead to explicit bounds on $\bar\theta$ and
$\bar\theta'$. While these bounds can possibly be improved, we do not pursue
here a goal of finding the sharpest bounds.
Clearly, statement (iv) is a corollary from (ii) and (iii), because
$|\hat Z^r(\infty)| \le \sum \hat Z_i^{r,+}(\infty) + \sum \hat Z_i^{r,-}(\infty)$, and then
\begin{align*}
\limsup_{r\rightarrow\infty} \E\exp\left(\theta |\hat Z^r(\infty)|\right)<\infty,
\end{align*}
for sufficiently small $\theta>0$ by application of the Cauchy-Schwartz inequality.
The proof of (i) will follow from a simple comparision to the infinite-server system.
The proof of (ii) will rely in essential way on statement (i).
Finally, the proof of (iii) will rely on statement (ii).

Theorem~\ref{th-main-no-abandonments} shows,
in particular, that in the Halfin-Whitt heavy traffic regime the steady state total
queue length $Q^r(\infty)$ scales like $O(\sqrt{r})$, because
\begin{align*}
\hat Q^r(\infty)&=(Z^r(\infty)-r-a\sqrt{r})^+/\sqrt{r}\\
&\le  (Z^r(\infty)-r)^+/\sqrt{r}\\
&=\hat Z^{r,+}(\infty).
\end{align*}
Furthermore, our result implies that
$\hat Q^r(\infty)$ decays at an exponential rate:
\begin{align*}
\pr\left(\hat Q^r(\infty)\ge x\right)\le c_1\exp(-c_2 x),
\end{align*}
for some constants $c_1,c_2>0$ uniformly in $r$.
The steady state number of idle servers,
$\left(r+a\sqrt{r}-Z^r(\infty)\right)^+$, also scales like $O(\sqrt{r})$,
because it is upper bounded by $a\sqrt{r} + \sum_i \hat Z_i^{r,-}(\infty)\sqrt{r}$.

The tightness of the sequence $\hat Z^r(\infty)$ implies the existence of at least one subsequential weak limit. Namely, there exists
a subsequence $r_n, n\in\Z_+$ and a random variable denoted by $\hat Z(\infty)$, such that $\hat Z^r(\infty)\Rightarrow\hat Z(\infty)$.
Such a limit is not necessarily unique as it may depend on the details of the scheduling policy.
Then, Theorem~\ref{th-main-no-abandonments} implies the following

\begin{cor}\label{coro:WeakLimitTight}
For a sufficiently small fixed $\theta>0$ and every weak limit $\hat Z(\infty)$ of the sequence $\hat Z^r(\infty)$,
$\E[\exp(\theta |\hat Z(\infty)|)]<\infty$.
\end{cor}

\begin{proof}
Fix positive $\theta \le \min(\bar \theta,\bar \theta')$.
Fix arbitrary $k>0$. If $\hat Z^r(\infty)$ converges weakly to $\hat Z(\infty)$, then by the Bounded Convergence Theorem,
\begin{align*}
\E\left[\min\left(\exp(\theta |\hat Z(\infty)|),k\right)\right]=\lim_{r\rightarrow\infty}
\E\left[\min\left(\exp(\theta |\hat Z^r(\infty)|),k\right)\right]\le \limsup_{r\rightarrow\infty}\E\left[\exp(\theta |\hat Z^r(\infty)|)\right],
\end{align*}
where the RHS is finite and does not depend on $k$. It remains to take the limit
of LHS on $k\to\infty$ and use the Monotone Convergence Theorem.
\end{proof}

Theorem~\ref{th-main-no-abandonments} leaves open the question whether (\ref{eq-est3333}) holds for all
$\theta$. (Equivalently whether (\ref{eq-est3331}) holds also when some $\nu_i>\mu_i$.) We conjecture
that it is the case, but we are currently  unable to verify this statement.

We now turn to the case when every customer class has strictly
positive abandonment rate: $\nu_i>0$. In this case, even tighter bounds
on any weak limit of the stationary distribution can be obtained.
It is known
that the presence of abandonments changes the properties of the tail of the distribution of $\hat Z^r(\infty)$. As an illustration
consider a single class model $M/M/N^r$, where the abandonment rate $\nu=\mu$ for simplicity. It is straightforward to see
that $Z(\infty)$ has the same distribution as in the $M/M/\infty$ system, namely it is Poisson with mean $\lambda r/\mu=r$.
As $r\rightarrow\infty$ the distribution of $\hat Z^r(\infty)$
then approaches a standard normal distribution. Thus we might expect the
tail $\pr(\hat Z^r(\infty)>x)$ to decay at the rate $\exp(-O(x^2))$. Namely, one might conjecture the following analogue of Theorem~\ref{th-main-no-abandonments}
- there exists a sufficiently
small $\theta>0$ such that $\sup_r\E[\exp(\theta (\hat Z^r(\infty))^2)]<\infty$. However, this cannot be the case - if random variable $G$ has  Poisson
distribution with mean $r>0$, then $\E[\exp(\theta G^2)]=\infty$ for every $\theta>0$.
Nevertheless, we can show that the following analogue of
Corollary~\ref{coro:WeakLimitTight} holds.

\begin{thm}
\label{th-main-abandonments}
For every subsequential weak limit $\hat Z(\infty)$
of $\hat Z^r(\infty), r\in\Z_+$, the following properties hold.\\
(i) If $\nu_i \le \mu_i$ for all $i$
(including possibly $\nu_i=0$ for some $i$), then
there exists $\bar\theta'>0$, such that
for every $\theta\in [0,\bar\theta']$
\begin{align}
\label{eq-est333-for-minus}
\E\exp\left(\theta \left(\hat Z^-(\infty)\right)^2\right)<\infty.
\end{align}
(ii) Suppose $\nu_i>0$ for all $i\in I$. Then
there exists $\bar\theta>0$ such that for
every $\theta\in [0,\bar\theta]$
\begin{align}
\label{eq-est333-for-plus}
\E\exp\left(\theta \left(\hat Z^+(\infty)\right)^2\right)<\infty.
\end{align}
\end{thm}
This result shows that, while the presence of abandonments does not change
the tail decay rate of $\hat Z^{r,+}(\infty)$ for each $r$,
the decay
rate of the tail of the weak limit $\hat Z^+(\infty)$
is faster -- it is Gaussian rather than simply exponential. This, in particular,
implies that
\begin{align*}
\pr\left(\hat Q(\infty)\ge x\right)\le c_3\exp(-c_4 x^2),
\end{align*}
for some constants $c_3,c_4>0$, where $\hat Q(\infty)$ is any weak limit of $\hat Q^r(\infty)$.

\section{Key bounds and monotonicity properties}
\label{section:key bounds}
We begin by establishing a few technical results. For the remainder of the paper we drop for convenience
the $(\infty)$ and superscript $r$
notations, and simply
use $Z,Z_i,\hat Z_i, \Phi_i$, etc. to denote our steady-state variables.
The same applies to the generator $A$ and operator $\bar A$ - the superscript $r$ is dropped.
\begin{lem}
\label{lemma:lowerbounds}
Suppose $\nu_i\le \mu_i$ for all $i$. Then for any $\theta\ge 0$,
\beql{eq-finite-exp222}
\limsup_{r\rightarrow\infty} \E \exp (\theta \hat Z_i^-) \le \exp (\rho_i \theta^2/2)+1,
\end{equation}
and
\beql{eq-finite-exp333}
\limsup_{r\rightarrow\infty} \E \exp (\theta \sum_{i} \hat Z_i^-) \le \prod_i [\exp (\rho_i \theta^2/2)+1].
\end{equation}
\end{lem}

\begin{proof}
We obtain the required lower bounds by comparing with a system with infinitely many servers.
Namely, in addition to the original system,
i.e. the process $X(t)$, which defines it,
consider the system with the same arrival rates,
same mean service times, same (exponential) service time distributions,
but with infinite number of servers.
Let $G_i$ be the steady-state number of customers
in service in this infinite server model.
Then, each $G_i$ has Poisson distribution with mean $\rho_i r$, and all
$G_i$'s are independent.

Given the
$\nu_i\le \mu_i$ assumption, clearly we can construct both the original and the infinite-server processes on a common probability space in a way such that, w.p.1
$G_i(t) \le Z_i(t)$ for all $t$.
That is, we can construct a ``joint'' Markov process,
for which the original process $X(t)$ and the infinite-server process
are different projections. Given that $X(t)$ is irreducible and positive recurrent,
we can make sure that the constructed joint process is
also irreducible and positive recurrent,
and thus has a unique stationary distribution.
Then, the relation $G_i \le Z_i$
holds in the stationary
regime as well.
Then also $\hat G_i \le \hat Z_i$, where
we denoted $\hat G_i=(G_i -\rho_i r)/\sqrt{r}$.
Since $G_i$ has Poisson distribution with parameter $\rho_i r$,
for any real $\theta$,
\beql{eq-finite-exp}
\E \exp (\theta \hat G_i) = \exp \left[ -\theta \rho_i \sqrt{r}
-\rho_i r (1-e^{\theta/\sqrt{r}}) \right]\stackrel{r\rightarrow\infty}{\longrightarrow} \exp (\rho_i \theta^2/2).
\end{equation}

Using $\hat{Z}_i\ge \hat{G}_i$, \eqn{eq-finite-exp}
immediately implies, that for any $\theta\ge 0$,
\beql{eq-finite-exp2}
\limsup_{r\rightarrow\infty} \E \exp (\theta \hat Z_i^-) \le
\limsup_{r\rightarrow\infty} \E \exp (\theta \hat G_i^-) \le
\exp (\rho_i \theta^2/2)+1,
\end{equation}
(where we used $\exp (\theta \hat G_i^-) \le 1+\exp (- \theta \hat G_i)$)
and
\beql{eq-finite-exp3}
\limsup_{r\rightarrow\infty} \E \exp (\theta \sum_i \hat Z_i^-) \le
\limsup_{r\rightarrow\infty} \E \exp (\theta \sum_i \hat G_i^-) \le
\prod_i [\exp (\rho_i \theta^2/2)+1].
\end{equation}
\end{proof}

Next we obtain  a bound which, roughly speaking, says that,
when all $\nu_i\le \mu_i$,
it is unlikely for $\hat\Phi$ to be large if $\hat Z$ is not large.
\begin{lem}
\label{lem-key-bound}
Suppose $\nu_i\le \mu_i$ for all $i$.
For arbitrary  $b\ge 0, \theta_1>0$, and $n\in\Z_+$,
there exists $C=C(b,\theta_1,n)$ such that
for all $\theta \in [0,\theta_1]$,
\beql{eq-key-bound}
\limsup_{r\rightarrow\infty} \E \left[I\{\hat{Z}\le b\} (\sum_i \hat Z_i^-)^n \exp(\theta \hat\Phi)\right] \le C.
\end{equation}
\end{lem}

\begin{proof}
Inequality
\beql{eq-ineq1}
(1/\mu_{\min}) \sum_i \hat Z_i^+ - (1/\mu_{\max}) \sum_i \hat Z_i^- \ge \hat\Phi
\end{equation}
holds always. Condition $\hat{Z}\le b$ is equivalent to
\beql{eq-ineq2}
\sum_i \hat Z_i^+ - \sum_i \hat Z_i^- \le b.
\end{equation}
Multiplying  \eqn{eq-ineq2} by $(1/\mu_{\min})$  and combining with  \eqn{eq-ineq1},
we get
\beql{eq-ineq3}
\hat\Phi \le (1/\mu_{\min} - 1/\mu_{\max}) \sum_i \hat Z_i^- + (1/\mu_{\min})b.
\end{equation}
Applying bound \eqn{eq-ineq3} to the left-hand side of \eqn{eq-key-bound}, and then using
\eqn{eq-finite-exp333}, we obtain \eqn{eq-key-bound}.
\end{proof}

Finally, we will need the following monotonicity result.
It says, roughly speaking, that for any well-defined system
(i.e., a process satisfying Assumption~\ref{assumption:Basic})
there always exists another well-defined system with smaller
abandonment rates, in which the number of customers of each class is larger.
It is important to note that Lemma~\ref{lem-monotonicity} does {\em not}
imply that this monotonicity works in the ``opposite direction'': namely,
it does {\em not} imply that for any system there exists another system with {\em larger}
abandonment rates, in which the number of customers of each class is {\em smaller}.

\begin{lem}
\label{lem-monotonicity}
Consider a process $X(t)$, satisfying Assumption~\ref{assumption:Basic}
for a given a set of parameters
$r,a,N,\mu_i,\nu_i, i\in I$. 
Suppose we have a modified set of parameters $r,a,N,\mu_i,\nu'_i, i\in I$,
with $\nu_i'\le \nu$ for all $i$.
Then there exists a Markov process $X'(t)$ with associated $Z_i'(t),\Psi_i'(t)$, 
satisfying Assumption~\ref{assumption:Basic}
with the modified parameter set, and such that the following holds.
The processes $X(t)$ and $X'(t)$ can be constructed on
a common probability space in a way such that, first, with probability $1$
$Z_i(t)\le Z'_i(t)$ and $\Psi_i(t)\le \Psi'_i(t)$ for all $i$ and $t$,
and, second, the joint process $(X(t),X'(t))$ is an irreducible 
positive recurrent Markov chain.
\end{lem}

Note that the lemma claim that the joint process $(X(t),X'(t))$
is an ergodic Markov chain, allows us to conclude that the steady-state versions
$X(\infty)$ and $X'(\infty)$ can be constructed on a common probability space
in a way such that $Z_i(\infty)\le Z'_i(\infty)$ and $\Psi_i(\infty)\le \Psi'_i(\infty)$ 
for all $i$.

Before giving a proof, we informally describe the simple intuition behind it. 
Given a queueing system corresponding to parameters $r,a,N,\mu_i,\nu_i, i\in I$,
operating under any scheduling policy so that the Assumption 1 is valid, consider a modified system constructed as follows.
Assume for simplicity that $\nu_i >0$ for each $i$.
Upon every abandonment instance of a system 
(which happens with rate $\nu_i$ for each class $i$ customer waiting in the queue), 
customers abandon only with probability $\nu'/\nu$ and with the remaining probability 
$(1-\nu'/\nu)$ 
join a new ''virtual'' class $i$ pool of customers, which we call class $i', i\in I$ customers, with some abuse of notations.
The decision is done independently for all customers.
We assume that the class
$i'$ customers for any $i\in I$ have lower priority than all the customers in classes $j\in I$ and priority mechanism is preemptive
resume. Thus the portion of the system corresponding to classes $j\in I$ only is the same as the original system. The service
times of class $i'$ customers are assumed to be exponentially distributed random variable with parameter $\mu_i$. 
Additionally, class $i'$ customers abandon with rate $\nu'_i$. 
It is easy to see that each queued classes $i$ customer
abandons with rate $\nu_i$; and each queued
customer of class $i$ or $i'$ 
abandons with rate $\nu'_{i}$, {\em if we do not count an $i$-to-$i'$ conversion
as an abandonment}.
This mechanism induces a new queueing system such 
$Z_{i}(t)\le Z'_{i}(t) \equiv Z_{i}(t)+Z_{i'}(t)$ (and similarly for $\Psi$). 

\begin{proof}
The states $X'(t)$ of the modified process are defined 
as augmented states $X(t)$, as follows:
$$
X'(t) = \left[X(t), (Z'_i(t), ~i\in I)\right],
$$
where
$(Z'_i(t), ~i\in I)$ is a vector such that $Z_i(t)\le Z'_i(t)$ for all $i$.
(Recall that $X(t)$ uniquely
determines the associated $Z_i(t), \Psi_i(t), ~i\in I$.
Therefore, corresponding to any state of $X(t)$ there is a countable set
of states of $X'(t)$.) Now,
consider an arbitrary fixed function, that maps 
$X'(t)$ into an integer vector $(\Psi'_i(t), ~i\in I)$, while satisfying
the following conditions: $\Psi_i(t)\le \Psi'_i(t)$ for all $i$;
$\sum_i \Psi'_i(t) = \min\{N, \sum_i Z'_i(t)\}$.
(The deterministic mapping of $X'(t)$ into $(Z'_i(t), ~i\in I)$ is the projection.)
Recall notation $Q_i(t)=Z_i(t)-\Psi_i(t)$, and denote $Q'_i(t)=Z'_i(t)-\Psi'_i(t)$.
Observe that we always have $Q_i(t)\le Q'_i(t)$; 
indeed, if $\sum_i \Psi_i(t) < N$ then all $Q_i(t)=0$, 
and if $\sum_i \Psi_i(t) = N$, 
then by construction $\sum_i \Psi'_i(t)\ge \sum_i \Psi_i(t) = N$, which coupled with $\sum_i \Psi'_i(t)\le N$
implies $\Psi_i(t)= \Psi'_i(t)$ for all $i$.

We will now define the transition rates of process $X'(t)$;
they will be such that the order relations $Z_i(t)\le Z'_i(t)$ are preserved.
We will use time indices $t-$ and $t$ for a state just before and just after a transition
at time $t$. The transitions will be of two types. The first type transitions are ``driven''
by those of process $X$, and have the same rates. When $X$ makes a transition 
from $X(t-)$ to $X(t)$, such that $Z_i(t)=Z_i(t-)$ for all $i$
(no ``arrivals'' or ``departures'' in the $X$-process), or $Z_i(t)=Z_i(t-)+1$
for one of the $i$ (an ``arrival'' in the $X$-process), 
then $X'(t)=(X(t), (Z'_i(t), ~i\in I))$, where each
$Z'_i(t)=Z'_i(t-) + Z_i(t)-Z_i(t-)$. (Simply put, in $X'$ we make the same change
of $X$-component and increment $Z'_i$'s in the same way.)
When, $X$ makes a transition 
from $X(t-)$ to $X(t)$, such that $Z_i(t)=Z_i(t-)-1$
for one of the $i$ (a ``departure'' in the $X$-process), 
say $i=j$, then $X'(t)=(X(t), (Z'_i(t), ~i\in I))$,
where $Z'_i(t)=Z'_i(t-)$ for all $i\ne j$, and $Z'_i(t)$ is set to
$Z'_i(t-)$ or $Z'_i(t-)-1$ with probabilities
$$
p=\frac{\nu_i Q_i(t-)}{\nu_i Q_i(t-)+\mu_i \Psi_i(t-)} \frac{\nu_i-\nu'_i}{\nu_i} 
~~\mbox{and}~~ 1-p,
$$
respectively. (In the expression for $p$ the $\nu_i$ in both the numerator
and denominator cancels out -- we left it there to make the meaning of 
probability $p$ transparent. In particular, the expression does {\em not}
require that $\nu_i>0$. We also do not need to ``worry'' about the 
case when $\nu_i Q_i(t-)+\mu_i \Psi_i(t-)=0$, because in this case
a ``departure of class $i$'' transition in $X$ is impossible,
and we only define $p$ for such transtions.)

The second type transitions are associated with each customer class $j$
(these are the ``departures'' in the $X'$-process only);
namely, with the rate $\nu'_j (Q'_j(t-)-Q_j(t-))+\mu_j (\Psi'_j(t-)-\Psi_j(t-))$
the transition form $X'(t-)$ to $X'(t)$ is such that
$X(t)=X(t-)$, $Z'_j(t)=Z'_j(t-)-1$, and $Z'_i(t)=Z'_i(t-)$ for all $i\ne j$.

The Markov chains $X'(t)$ and (its projection) $X(t)$ are constructed on the 
same probability space. The order relations $Z_i(t)\le Z'_i(t)$ are preserved
for all $i$ at all times, by construction. 
By construction the process $X(t)$ is defined ``as itself''.
It it easily checked that the process
$X'(t)$ satisfies
Assumption~\ref{assumption:Basic},
including the irreducibility condition.
Therefore, $X'(t)$ is positive recurrent.
Since $X(t)$ is a projection of $X'(t)$, the process $X'(t)$
itself can be taken as the joint process $(X(t),X'(t))$.
The proof is complete
\end{proof}

\section{Proof of Theorem~\ref{th-main-no-abandonments}(i) and (ii)}
\label{section:mainresult1}

Statement (i) follows from Lemma~\ref{lemma:lowerbounds}.

To prove statement (ii) it suffices by to Lemma~\ref{lem-monotonicity}
to consider the case when all $\nu_i=0$, which is what we assume for the
rest of the proof.
Consider the process in the stationary regime, i.e. $\mb{x}$ is
random with the distribution
equal to the stationary distribution $\pi$
of the Markov chain. First, for every state $\mb{x}\in\mathcal{X}$ we have
$$
\bar{A} \hat\Phi (\mb{x}) = \frac{1}{\sqrt{r}}
[\sum_i (\lambda_i r)/\mu_i - \sum_i \Psi_i(\mb{x})] = - \hat{Z}_a(\mb{x}),
$$
where $\hat{Z}_a(\mb{x}) = \min\{\hat{Z}(\mb{x}),a\}$. We remark that this expression does {\em not} require chain $X(\cdot)$ to be in steady-state.

Consider a (candidate) Lyapunov function $\exp(\theta \hat\Phi(\mb{x}))$ for every state $\mb{x}\in\mathcal{X}$, where  $\theta>0$ is for a moment arbitrary.
We use the Taylor expansion
\begin{align*}
\exp(\theta y)=\exp(\theta x)(1+\theta(y-x)+(1/2)\theta^2(y-x)^2\exp(\theta z)),
\end{align*}
for some $z$ satisfying $|z|\le |y-x|$.
Then $\exp(\theta y)\le \exp(\theta x)(1+\theta(y-x)+(1/2)\theta^2(y-x)^2\exp(\theta |y-x|))$.
We will apply this for $x=\hat\Phi(\mb{x}),y=\hat\Phi({\bf y})$, where
${\bf y} \in T(\mb{x})$.
We obtain
\begin{align}
\label{eq-long}
\bar{A} \exp(\theta \hat\Phi(\mb{x})) \le \exp(\theta \hat\Phi(\mb{x}))\Big[\theta \bar{A} \hat\Phi (\mb{x})
+(1/2)\theta^2 [\sum_i \lambda_i r + (r+a\sqrt{r})\mu_{\max}] (\mu_{\min}^{-1}/\sqrt{r})^2 \exp(\theta \mu_{\min}^{-1}/\sqrt{r})\Big]
\end{align}
where we use the fact that $|y-x| \le
1/(\mu_{\min}\sqrt{r})$, and $\sum_i \lambda_i r + (r+a\sqrt{r})\mu_{\max}$ is a crude
upper bound on the maximum transition rate (involving a customer arrival or
departure).
We obtain
\begin{align}
\label{eq-drift}
\bar{A} \exp(\theta \hat\Phi(\mb{x})) \le \exp(\theta \hat\Phi(\mb{x}))
\Big[- \theta \hat{Z}_a(\mb{x})
+(1/2)\theta^2 c_1 \Big],
\end{align}
where, uniformly on all $r\ge 1$ and all $\theta\le 1$,
\begin{align*}
c_1=\left(\sum_i \lambda_i + (1+a)\mu_{\max}\right)\mu_{\min}^{-2}\exp(\mu_{\min}^{-1}))\ge
\left(\sum_i \lambda_i + (1+a/\sqrt{r})\mu_{\max}\right)\mu_{\min}^{-2}\exp(\theta \mu_{\min}^{-1}/\sqrt{r}).
\end{align*}
In the rest of this proof we only consider $\theta \le 1$.

Let $\hat\Phi_{(k)}(\mb{x})\triangleq \min\{\hat\Phi(\mb{x}), k\}$.
Clearly, $\exp(\theta \hat\Phi_{(k)}(\mb{x}))$ is constant everywhere
except a finite subset of state space $\mathcal{X}$,
because (Assumption~\ref{assumption:Basic})
there is only a finite number of states ${\bf x}\in \mathcal{X}$
for which $\sum \hat Z_i(\mb{x})/\mu_i <k$.
Therefore, \eqn{eq:ExpectedGeneratorZero}
holds for the function $\exp(\theta \hat\Phi_{(k)}(\mb{x}))$.

Observe that $\hat \Phi(\mb{x}) \ge k$ implies
$\bar A \exp(\theta \hat\Phi_{(k)}(\mb{x})) \le 0$,
and $\hat \Phi(\mb{x}) < k$ implies
$\bar A \exp(\theta \hat\Phi_{(k)}(\mb{x}))
\le \bar A \exp(\theta \hat\Phi(\mb{x}))$.
Then, from~\eqn{eq-drift} we obtain:
\begin{align}
\bar A\exp(\theta \hat\Phi_{(k)}(\mb{x}))
&\le I\{\hat\Phi(\mb{x})<k\} I\{\hat Z(\mb{x}) \ge a\} \exp(\theta \hat\Phi(\mb{x}))\Big[-\theta a+(1/2)\theta^2 c_1\Big] \label{eq-drift-tr2}\\
&+I\{\hat Z(\mb{x}) < 0\} \exp(\theta \hat\Phi(\mb{x}))\Big[-\theta \hat Z(\mb{x}) +(1/2)\theta^2 c_1\Big]  \label{eq-drift-tr3}\\
&+I\{0\le \hat Z(\mb{x}) < a\} \exp(\theta \hat\Phi(\mb{x}))\Big[(1/2)\theta^2 c_1\Big]. \label{eq-drift-tr4}
\end{align}
For the process in steady state, all terms \eqn{eq-drift-tr2}-\eqn{eq-drift-tr4}
have finite expected values for each $k$ and $r$, including the left-hand side of (\ref{eq-drift-tr2}). Indeed, the left- and right-hand sides
of
\eqn{eq-drift-tr2} are bounded, and \eqn{eq-drift-tr3}
and \eqn{eq-drift-tr4} are non-negative and
have finite expectations by Lemma~\ref{lem-key-bound}. We now take the expectation
w.r.t. the stationary distribution $\pi$ on $\mathcal{X}$ of both sides,
 use
$\E[\bar A\exp(\theta \hat\Phi_{(k)}(\mb{x}))]=0$ (which follows from (\ref{eq:ExpectedGeneratorZero})), and cancel $\theta$.
Recalling our notation $\hat\Phi=\hat \Phi^r(\infty)$, which is the same as $\hat\Phi(\mb{x})$, where
$\mb{x}$ is distributed according to the stationary measure $\pi$, we obtain
\begin{align}
\label{eq-drift-tr22}
\E \left[I \{\hat\Phi<k\} I\{\hat Z \ge a\} \exp(\theta \hat\Phi) (a-(1/2)\theta c_1)\right]
\le
\end{align}
\begin{align}
\label{eq-drift-tr33}
\E \left[I\{\hat Z < 0\} \exp(\theta \hat\Phi)\Big[-\hat Z+(1/2)\theta c_1\Big]\right] +
\end{align}
\begin{align}
\label{eq-drift-tr44}
\E \left[I\{0\le \hat Z < a\} \exp(\theta \hat\Phi)\Big[(1/2)\theta c_1\Big]\right].
\end{align}
Now, let
$\bar\theta=\min\{a/c_1,1\}$.
Then
By Lemma~\ref{lem-key-bound}, we have
\begin{align*}
&\limsup_r\E \left[I\{\hat Z < 0\} \exp(\theta \hat\Phi)\Big[-\hat Z+(1/2)\theta c_1\Big]\right] \le C_2(a,c_1),\\
&\limsup_r\E \left[I\{0\le \hat Z < a\} \exp(\theta \hat\Phi)\Big[(1/2)\theta c_1\Big]\right]\le C_2(a,c_1),
\end{align*}
for some $C_2(a,c_1)<\infty$.
On the other hand $\bar\theta=a/c_1$ implies
that for any $\theta\in [0,\bar\theta]$,
$a-(1/2)\theta c_1>a/2>0$.
Thus
\begin{align*}
\limsup_r\E \left[I \{\hat\Phi<k\} I\{\hat Z \ge a\} \exp(\theta \hat\Phi)\right]\le C_3(a,c_1)<\infty,
\end{align*}
for some $C_3(a,c_1)$. Observe that this bound does not depend on $k$. Thus fix $r_0$ large enough so that for all $r\ge r_0$
\begin{align*}
\E \left[I \{\hat\Phi<k\} I\{\hat Z \ge a\} \exp(\theta \hat\Phi)\right]\le 2C_3(a,c_1).
\end{align*}
By Monotone Convergence Theorem, taking $k\rightarrow\infty$, we obtain
\begin{align*}
\E \left[ I\{\hat Z \ge a\} \exp(\theta \hat\Phi)\right]\le 2C_3(a,c_1),
\end{align*}
namely
\begin{align*}
\limsup_r\E \left[ I\{\hat Z \ge a\} \exp(\theta \hat\Phi)\right]<\infty.
\end{align*}
Again applying Lemma~\ref{lem-key-bound}, we have
\begin{align*}
\limsup_r\E \left[ I\{\hat Z < a\} \exp(\theta \hat\Phi)\right]<\infty.
\end{align*}
Combining, we obtain $\limsup_r\E \left[\exp(\theta \hat\Phi)\right]<\infty.$

To complete the proof of the theorem, take any $\theta<\bar\theta$ and
$p=p(\theta)>1$ and such that $p\theta < \bar\theta$; and then $q$ such that $1/p+1/q=1$.
Using Holder inequality, we have
$$
\E  \exp (\theta \sum_i (1/\mu_i) \hat Z_i^+)
= \E \left[ \exp(\theta \hat\Phi)  \exp (\theta \sum_i (1/\mu_i) \hat Z_i^-)  \right]
\le \left[\E \exp(p\theta \hat\Phi) \right]^{1/p}
\left[\E \exp (q\theta \sum_i (1/\mu_i) \hat Z_i^-) \right]^{1/q}.
$$
By Lemma~\ref{lemma:lowerbounds} we have
\begin{align*}
\limsup_r\E\left[ \exp (q\theta \sum_i (1/\mu_i) \hat Z_i^-)\right] <\infty,
\end{align*}
from which we obtain
\begin{align*}
\limsup_r\E\left[  \exp (\theta \mu_{\max}^{-1}\sum_i  \hat Z_i^+)\right]<\infty.
\end{align*}
This proves \eqn{eq-est3332}. The proof of
Theorem~\ref{th-main-no-abandonments}(ii) is complete.

\section{Proof of Theorem~\ref{th-main-no-abandonments}(iii)}
\label{section:mainresult13333}

\subsection{Auxilliary results}

\begin{lem}
\label{lem-sum-Abar-zero}
For any $\theta\ge 0$ and any $r$, the stationary distribution
is such that
\beql{eq-mean-abar-zero}
\E \bar A \exp(-\theta \hat\Phi) = \sum_{\mb{x}} \pi(\mb{x}) \bar A \exp(-\theta \hat\Phi(\mb{x}))
= 0.
\end{equation}
\end{lem}

\begin{proof}
The case $\theta=0$ is trivial.
For the rest of the proof, consifer a {\em fixed pair} of
$\theta>0$ and $r$. (The chosen constants may depend on $\theta$ and $r$.)

Consider a fixed $c>0$ and any state $\mb{x}$ such that $\hat\Phi(\mb{x})\ge c$. Then,
it is easy to observe that for some fixed $c_{20}>0$, $c_{21}>0$, and any $k\ge c$,
\beql{eq-Abar-bound}
\left| \bar A \exp(-\theta \hat\Phi(\mb{x})) \right|
\le (c_{20} \hat\Phi(\mb{x}) + c_{21}) \exp(-\theta \hat\Phi(\mb{x})),
\end{equation}
\beql{eq-Abar2-bound}
\left| \bar A \exp(-\theta \hat\Phi_{(k)}(\mb{x})) \right|
\le (c_{20} \hat\Phi(\mb{x}) + c_{21}) \exp(-\theta \hat\Phi(\mb{x})).
\end{equation}
Indeed, the total rate
of all transitions out of state $\mb{x}$ that may change the value of $\hat\Phi$
is upper bounded by $c_{20} \hat\Phi(\mb{x}) + c_{21}$ for some constants
$c_{20}>0, c_{21}>0$. (From $\hat \Phi = \sum_i [Z_i - \rho_i r]/(\mu_i \sqrt{r})$
we get an upper bound on $Z$, linear in $\hat \Phi$, which gives the bound
on the customer departure rate; the arrival rate is constant.)
Also $\hat\Phi(\mb{y}) \ge \hat\Phi(\mb{x}) - 1/(\mu_{\min}\sqrt{r})
\ge \hat\Phi(\mb{x}) - 1/\mu_{\min}$
for any state $\mb{y}\in T(\mb{x})$.
Therefore,
$$
\left| \bar A \exp(-\theta \hat\Phi(\mb{x})) \right|
\le (c_{20} \hat\Phi(\mb{x}) + c_{21}) 2 \exp[-\theta (\hat\Phi(\mb{x})-1/\mu_{\min})],
$$
where the factor $2\exp(\theta/\mu_{\min})$ can be dropped by rechoosing
$c_{20}, c_{21}$. This proves \eqn{eq-Abar-bound}, and \eqn{eq-Abar2-bound}
is proved the same way.

From \eqn{eq-Abar-bound} and \eqn{eq-Abar2-bound} we conclude
that the values of
$$
\left| \bar A \exp(-\theta \hat\Phi(\mb{x})) \right|
~\mbox{and}~
\left| \bar A \exp(-\theta \hat\Phi_{(k)}(\mb{x})) \right|
$$
are uniformly bounded for all $\mb{x}$ and $k$. (Recall that the
number of states with $\hat\Phi(\mb{x})< c$ is finite, for any $c$.)
This means, in particular, that the sum in \eqn{eq-mean-abar-zero}
converges absolutely.

We know that
\beql{eq12345}
\E \bar A \exp(-\theta \hat\Phi_{(k)})
= 0,
\end{equation}
by \eqn{eq:ExpectedGeneratorZero}, because
$\exp(-\theta \hat\Phi_{(k)})$
is constant outside a finite set of states. The LHS of \eqn{eq12345}
can be written as
\beql{eq12345-2}
\E \bar A \exp(-\theta \hat\Phi_{(k)}) =
\sum_{B_1(k)} \pi(\mb{x}) \bar A \exp(-\theta \hat\Phi(\mb{x}))
+ \sum_{B_2(k)} \pi(\mb{x}) \bar A \exp(-\theta \hat\Phi_{(k)}(\mb{x})),
\end{equation}
where $B_1(k)$ is the set of states $\mb{x}$ such that
$\mb{x}$ and all $\mb{y}\in T(\mb{x})$, are within set $B(k)=\{w~|~\hat\Phi(w)<k\}$
(so that $B_1(k) \subseteq B(k)$;
$B_2(k)$ is the set of ``boundary'' states $\mb{x}$ for which
there exists $\mb{y}\in T(\mb{x})$ such that
either $\mb{x}\in B(k)$ or $\mb{y}\in B(k)$, but not both.
As $k\to\infty$, the sum over $B_1(k)$ converges to the sum in \eqn{eq-mean-abar-zero},
and the
sum over $B_2(k)$ vanishes. This proves \eqn{eq-mean-abar-zero}.
\end{proof}

\begin{lem}
\label{lem-key-bound55}
For arbitrary  $b\ge 0$ and $n\in\Z_+$, there exist $\theta_2=\theta_2(b,n)>0$
and $C=C(b,n)$ such that
for all $\theta \in [0,\theta_2]$,
\beql{eq-key-bound55}
\limsup_{r\rightarrow\infty} \E \left[I\{\hat{Z}\ge -b\} (\sum_i \hat Z_i^+)^n \exp(-\theta \hat\Phi)\right] \le C.
\end{equation}
\end{lem}

\begin{proof}
The proof is analogous to that of Lemma~\ref{lem-key-bound},
except instead of relying on the bound \eqn{eq-finite-exp333},
which involves $Z_i^-$ and holds when all $\nu_i\le\mu_i$,
we will use bound \eqn{eq-est3332}, which involves $Z_i^+$ and holds
for arbitrary $\nu_i\ge 0$.
Namely,
from $\hat{Z}\ge -b$ we obtain
\beql{eq-ineq355}
-\hat\Phi \le (1/\mu_{\min} - 1/\mu_{\max}) \sum_i \hat Z_i^+ + (1/\mu_{\min})b
\end{equation}
similarly to \eqn{eq-ineq3}.
Then applying bound \eqn{eq-ineq355} to the left-hand side of \eqn{eq-key-bound55}, and then using \eqn{eq-est3332}, we obtain \eqn{eq-key-bound55}.
\end{proof}

\subsection{Proof of Theorem~\ref{th-main-no-abandonments}(iii)}

For every state $\mb{x}\in\mathcal{X}$ we have
\begin{align*}
\bar{A} \hat\Phi (\mb{x}) &= \frac{1}{\sqrt{r}}
[\sum_i (\lambda_i r)/\mu_i - \sum_i \Psi_i(\mb{x}) -\sum_i \nu_i Q_i(\mb{x}] =
- \hat{Z}_a(\mb{x}) -\sum_i \nu_i \hat Q_i(\mb{x}),
\end{align*}

Consider a (candidate) Lyapunov function $\exp(-\theta \hat\Phi(\mb{x}))$,
with $\theta>0$. Then, analogously to \eqn{eq-long}
we obtain
\begin{align}
\label{eq-long55}
\bar{A} \exp(-\theta \hat\Phi(\mb{x})) &\le \exp(-\theta \hat\Phi(\mb{x}))\Big[-\theta \bar{A} \hat\Phi (\mb{x})\\
&+(1/2)\theta^2 [\sum_i \lambda_i r + (r+a\sqrt{r})\mu_{\max} +\nu_{\max} \sum_i Q_i(\mb{x})]
(\mu_{\min}^{-1}/\sqrt{r})^2 \exp(\theta \mu_{\min}^{-1}/\sqrt{r})\Big], \notag
\end{align}
where the crude estimate of the maximum total transition rate in the $\theta^2$ term
accounts for the transitions due to abandonments as well.
In turn, \eqn{eq-long55} can be rewritten as
\begin{align}
\label{eq-drift5566}
\bar{A} \exp(-\theta \hat\Phi(\mb{x}))
\le \exp(-\theta \hat\Phi(\mb{x}))
\Big[ \theta \hat{Z}_a(\mb{x}) + \theta \sum_i \nu_i \hat Q_i(\mb{x})
+\theta^2 [c_{11}+c_{12} \hat Z^+] \Big],
\end{align}
where $c_{11}>0$ and $c_{12}>0$ are constants (independent or $r$),
and the inequality holds for all sufficiently large $r$.
Finally, we fix arbitrary $b>0$ and rewrite \eqn{eq-drift5566} as
\begin{align}
\label{eq-drift5577}
\bar{A} \exp(-\theta \hat\Phi(\mb{x}))
\le I\{\hat Z <-b\}
\exp(-\theta \hat\Phi(\mb{x}))
\Big[ \theta (-b)
+\theta^2 c_{11} \Big] +
\end{align}
\begin{align}
\label{eq-drift5588}
I\{\hat Z \ge -b\}
\exp(-\theta \hat\Phi(\mb{x}))
\Big[ \theta a + \theta c_{13} \sum_i \hat Z_i^+(\mb{x})
+\theta^2 [c_{11}+c_{14} \sum_i \hat Z_i^+(\mb{x})] \Big],
\end{align}
where $c_{13}>0$ and $c_{14}>0$ are constants independent or $r$.

We now consider our system in steady state.
The LHS of \eqn{eq-drift5577} has expected value $0$ by Lemma~\ref{lem-sum-Abar-zero}.
For all sufficiently small $\theta>0$, \eqn{eq-drift5588} has 
expectation bounded uniformly in $r$ by Lemma~\ref{lem-key-bound55}. Further, let us restrict
ourselves only to small values of $\theta$ such that
$b-\theta c_{11} >c_{15}$ for an arbitrary fixed constant $0< c_{15}<b$.
Then, the RHS of \eqn{eq-drift5577} has a well defined, strictly negative
expected value. (At this point we do not claim that the
 expectation is finite -- that will follow shortly.)
We take the expectation of both parts of \eqn{eq-drift5577}-\eqn{eq-drift5588},
cancel $\theta$, move (what is left of) the term \eqn{eq-drift5577}
to the LHS of the inequality, and divide the inequality by
$b-\theta c_{11}>c_{15}>0$.
Letting $r\to\infty$ and
applying Lemma~\ref{lem-key-bound55}
gives us
\begin{align*}
\limsup_r\E \left[ I\{\hat Z < -b\} \exp(-\theta \hat\Phi)\right]<\infty.
\end{align*}
Again applying Lemma~\ref{lem-key-bound55}, we have
\begin{align*}
\limsup_r\E \left[ I\{\hat Z \ge -b\} \exp(-\theta \hat\Phi)\right]<\infty.
\end{align*}
Combining, we obtain $\limsup_r\E \left[\exp(-\theta \hat\Phi)\right]<\infty.$
From this, the required bound \eqn{eq-est3333}
for sufficiently small $\theta>0$,
follows by using the identity
$\sum_i \hat Z_i^-/\mu_i = -\hat\Phi+\sum_i \hat Z_i^+/\mu_i$,
property \eqn{eq-est3332}, and Holder inequality.
The proof is complete.

\section{Proof of Theorem~\ref{th-main-abandonments}}
\label{section:mainresult2}

Statement (i) of the theorem follows from the following result,
which is  a
stronger version of
Lemma~\ref{lemma:lowerbounds}.
\begin{lem}
\label{lem-bound-on-square}
Suppose $\nu_i\le \mu_i$ for all $i$.
There exists a sufficiently small $\theta>0$ such that, for each $i$,
\beql{eq-bound-on-square}
\limsup_{r\rightarrow\infty} \E\left[\exp(\theta (\hat Z_i^-)^2)\right]<\infty.
\end{equation}
\end{lem}

Given the fact that the vector $(\hat Z_i^-, ~i\in I)$ in Lemma~\ref{lem-bound-on-square}
is stochastically dominated by $(\hat G_i^-, ~i\in I)$,
where $\hat G_i = (G_i-\rho_i r)/\sqrt{r}$, and $G_i$'s are independent r.v.
with Poisson distribution
with mean $\rho_i r$, this lemma in turn is a corollary of the following asymptotic
property of Poission distributions.
\begin{lem}
\label{lem-bound-on-Gminus-square}
Consider a family of Poisson distributed random variables $H(p)$ with mean $p>0$.
Denote
$$
h_n=h_n(p)=\pr\{H(p)=n\}=e^{-p}p^n/n!, n\in\Z_+.
$$
Then, there exists a constant $C>0$,
such that for all sufficiently large $p$ and all $n\le p$,
\beql{eq-gaussian-bound}
h_n \le C \frac{1}{\sqrt{p}} e^{-(n-p)^2/(2p)}.
\end{equation}
Consequently, for any $0\le \theta < 1/2$,
\beql{eq-bound-on-Gminus-square}
\limsup_{p\rightarrow\infty}
\E e^{\theta[(H(p)-p)^-]^2/p} <\infty.
\end{equation}
\end{lem}


\begin{proof}
Denote $m=\lfloor p \rfloor$. Using Stirling formula for the factorial
in $h_m=e^{-p}p^m/m!$, we directly verify that
$$ 
\lim_{p\to\infty}
h_m \left[\frac{1}{\sqrt{2\pi}} \frac{1}{\sqrt{p}}\right]^{-1} = 1,
$$ 
in particular, for some constant $C_1>0$ and all large $p$
\beql{eq-hm2}
h_m \le C_1 \frac{1}{\sqrt{p}}.
\end{equation}
Therefore, to show \eqn{eq-gaussian-bound}, it will suffice
to show the existence of $C_2>0$ such that for all large $p$ and any
$n\le m$, we have
\beql{eq-hn-ratio}
h_n/h_m \le C_2 e^{-(n-p)^2/(2p)}.
\end{equation}
We only need to verify \eqn{eq-hn-ratio} for $n<m$, because
$(m-p)^2/(2p)\to 0$ as $p\to\infty$, so that the case $n=m$ can be covered
by rechoosing (increasing) $C_2$ if necessary.
We have
$$
\log (h_n/h_m) = \log \prod_{j=n+1}^m \frac{j}{p}
= \sum_{j=n+1}^m \log (1-\frac{p-j}{p})
\le - \sum_{j=n+1}^m \frac{p-j}{p}
= -\frac{1}{2p}(m-n)(p-m +p-n-1).
$$
Note that $\epsilon=p-m \in [0,1)$. Substituting $m=p-\epsilon$
in the RHS of the above inequality, we see that it is
$$
-\frac{1}{2p}(p-n)^2 + \frac{1}{2p}(p-n+\epsilon(1-\epsilon))
\le -\frac{1}{2p}(p-n)^2 + 1
$$
for all $p\ge 1$. We proved \eqn{eq-gaussian-bound}.

To show \eqn{eq-bound-on-Gminus-square}, we can write
$$
\E e^{\theta[(H(p)-p)^-]^2/p} \le 1 + \sum_{n<p} h_n e^{\theta(n-p)^2/p} \le
$$
$$
1 + \sum_{n<p-1} C \frac{1}{\sqrt{p}}  e^{(\theta-1/2)(n-p)^2/p}
+ C \frac{1}{\sqrt{p}}.
$$
It remains to observe that
$$
\sum_{n<p-1} \frac{1}{\sqrt{p}}  e^{(\theta-1/2)(n-p)^2/p}
\le \int_{-\infty}^0 e^{-(1/2-\theta)\xi^2} d\xi.
$$
\end{proof}


Thus, statement (i) of the theorem is proved.
In view of Lemma~\ref{lem-monotonicity}, to prove statement (ii),
it suffices to consider the case $\nu_i\le \mu_i$ for all $i$.
Then the following proposition will imply
statement (ii) almost immediately.

\begin{prop}
\label{prop:abandon}
Suppose, $0<\nu_i\le \mu_i$ for all $i$.
There exists $\bar\theta>0$ and $c_0>0$ such that for all $\theta\in [0,\bar\theta]$ and $k$
\beql{eq-abandon}
\limsup_{r\rightarrow\infty} \E\left[I\lbrace\hat\Phi\le k\rbrace\exp(\theta \hat\Phi^2)\right]\le c_0,
\end{equation}
\beql{eq-abandon2}
\limsup_{r\rightarrow\infty} \E\left[I\lbrace\hat\Phi\le k\rbrace\exp(\theta (\sum_i \hat Z_i^+/\mu_i)^2)\right]\le c_0.
\end{equation}
\end{prop}

\begin{proof}[Proof of Proposition~\ref{prop:abandon}]
Property \eqn{eq-abandon2} follows from Lemma~\ref{lem-bound-on-square},
\eqn{eq-abandon}  and application of Cauchy-Schwartz inequality because
$$
(\sum_i \hat Z_i^+/\mu_i)^2 = (\hat\Phi +\sum_i \hat Z_i^-/\mu_i)^2
\le 2 \hat\Phi^2 + 2|I|\sum_i (\hat Z_i^-/\mu_i)^2.
$$
Thus, the main bulk of the proof is devoted to \eqn{eq-abandon}.

We first write the expression and derive bounds for $\bar{A} \hat\Phi$.
We have (using notation $\Psi=\sum_i \Psi_i$) for every state $\mb{x}$
\begin{align}\label{eq:drift-abandonments}
\bar{A} \hat\Phi (\mb{x}) &= \frac{1}{\sqrt{r}}[\sum_i (\lambda_i r)/\mu_i - \sum_i \Psi_i(\mb{x})-\sum_i(Z_i(\mb{x})-\Psi_i(\mb{x}))\nu_i/\mu_i]\\
&\le \frac{1}{\sqrt{r}}[r - \Psi(\mb{x})-\nu_{\min}/\mu_{\max}(Z(\mb{x})-\Psi(\mb{x}))]
\notag\\
&=-\min\lbrace\hat Z(\mb{x}),a\rbrace-\nu_{\min}/\mu_{\max}(\hat Z(\mb{x})-a)^+.\notag
\end{align}
Now we have
\begin{align*}
&\hat Z(\mb{x})=\sum_i \hat Z_i(\mb{x})=\sum_i \hat Z_i^+(\mb{x}) -\sum_i \hat Z_i^-(\mb{x})
\ge \mu_{\min} \hat\Phi^+(\mb{x}) - \sum_i \hat Z_i^-(\mb{x}) .
\end{align*}
Since $(x-a)^+$ is a monotone function, we obtain a bound
\begin{align}
\bar{A} \hat\Phi (\mb{x})
&\le -\min\lbrace\hat Z(\mb{x}),a\rbrace-(\nu_{\min}\mu_{\min}/\mu_{\max})\hat\Phi^+(\mb{x})
+(\nu_{\min}/\mu_{\max})\sum_i\hat Z_i^-(\mb{x})+(\nu_{\min}/\mu_{\max})a \notag\\
&\triangleq -\min\lbrace\hat Z(\mb{x}),a\rbrace-c_1\hat\Phi^+ +c_2\sum_i\hat Z_i^-(\mb{x})+c_3. \label{eq:BoundGenerator}
\end{align}

Similarly to \eqn{eq:BoundGenerator} we can also obtain a lower bound
on $\bar{A} \hat\Phi$. Namely, from \eqn{eq:drift-abandonments} we get
$$
\bar{A} \hat\Phi (\mb{x}) \ge 
-\min\lbrace\hat Z(\mb{x}),a\rbrace-\nu_{\max}/\mu_{\min}(\hat Z(\mb{x})-a)^+,
$$
which, using 
$\hat Z(\mb{x})\le \sum_i \hat Z_i^+(\mb{x})$, implies
\beql{eq:BoundGeneratorLower}
\bar{A} \hat\Phi (\mb{x}) \ge 
-\min\lbrace\hat Z(\mb{x}),a\rbrace-c_{25}\sum_i\hat Z_i^+(\mb{x}),
\end{equation}
where $c_{25}=\nu_{\max}/\mu_{\min}$.

Consider a (candidate) Lyapunov function $\exp(\theta \hat\Phi^2(\mb{x}))$ with $0<\theta\le \bar\theta$
where $\bar\theta$ will be implicitly defined below.
(Below we use the phrase ``for sufficiently small $\theta>0$'' to mean ``for all
$\theta \in (0,\bar\theta]$, with $\bar\theta$ rechosen to be smaller, if necessary.'')
Without the loss of generality we assume $\bar\theta\le 1$.
We have for $h(x)=\exp(\theta x^2)$ that $dh/dx=2\theta x\exp(\theta x^2)$ and $d^2 h/dx^2=2\theta \exp(\theta x^2)+4\theta^2x^2\exp(\theta x^2)$.
We use Taylor expansion
\begin{align*}
\exp(\theta y^2)=\exp(\theta x^2)\left(1+2\theta x(y-x)+\frac{1}{2}(2\theta+4\theta^2 z^2)(y-x)^2\exp(\theta z^2-\theta x^2)\right)
\end{align*}
for some $z$ satisfying $|z-x|\le |y-x|$.
We have
$z^2-x^2=(z-x)(z+x)\le 2|y-x|(|y|+|x|)$.
Then
\begin{align*}
\exp(\theta y^2)&\le \exp(\theta x^2)\left(1+2\theta x(y-x)
+(\theta+2\theta^2(|x|+|y-x|)^2(y-x)^2)\exp(\theta 2|y-x|(|y|+|x|))\right) \\
\end{align*}
Let $\hat\Phi_{(k)}=\min(\hat\Phi,k)$ and let us apply the above
inequalities for $x=\hat\Phi_{(k)}(\mb{x})$ and $y=\hat\Phi_{(k)}({\bf y})$,
where ${\bf y}\in T(\mb{y})$.
We have $|\hat\Phi({\bf y})-\hat\Phi(\mb{x})|\le \mu_{\min}^{-1}/\sqrt{r}$ implying the same for $\hat\Phi_{(k)}$.
Then
\begin{align*}
\max\{\hat\Phi_{(k)}({\bf y}),\hat\Phi_{(k)}(\mb{x})\} \le \mu_{\min}^{-1}/\sqrt{r}+k,
\end{align*}
and
\begin{align*}
\min\{\hat\Phi_{(k)}({\bf y}),\hat\Phi_{(k)}(\mb{x})\}\ge -\hat\Phi^-(\mb{x})-\mu_{\min}^{-1}/\sqrt{r}.
\end{align*}
This implies
\begin{align*}
2|\hat\Phi_{(k)}({\bf y})-\hat\Phi_{(k)}(\mb{x})| (|\hat\Phi_{(k)}({\bf y})|+|\hat\Phi_{(k)}(\mb{x})|)
\le 4[\mu_{\min}^{-2}/r+\mu_{\min}^{-1}k/\sqrt{r}+\mu_{\min}^{-1}\hat\Phi^-(\mb{x})/\sqrt{r}].
\end{align*}
For sufficiently large $r$ which depends on $k$ and model parameters,
we obtain even a cruder upper bound
\begin{align*}
2|\hat\Phi_{(k)}({\bf y})-\hat\Phi_{(k)}(\mb{x})| (|\hat\Phi_{(k)}({\bf y})|+|\hat\Phi_{(k)}(\mb{x})|)
\le \hat\Phi^-(\mb{x}) + 1
\end{align*}
giving
\begin{align*}
\exp(\theta 2|\hat\Phi_{(k)}({\bf y})-\hat\Phi_{(k)}(\mb{x})| (|\hat\Phi_{(k)}({\bf y})|+|\hat\Phi_{(k)}(\mb{x})|))
\le \exp(\theta\hat\Phi^-(\mb{x})+\theta).
\end{align*}
Note that, given $k$, the value of $\hat\Phi_{(k)}(\mb{x})$ can change
after a direct transition
only when $\sum_i \hat Z_i(\mb{x})/\mu_i\le k + 1/\mu_{min}$, implying
$Z\le r+\mu_{\max}(k+1/\mu_{min})\sqrt{r}$. Therefore,
 uniformly on all states $\mb{x}$,
the total rate of all transitions that can possibly change the
value of $\hat\Phi_{(k)}(\mb{x})$ is upper bounded by
\begin{align*}
r\sum_i\lambda_i+(r+a\sqrt{r})\mu_{\max}+\sqrt{r}\mu_{\max}(k+1/\mu_{min})\sum_i\nu_i\le c_4 r,
\end{align*}
for the appropriate choice of positive constant $c_4$.
Also,
\begin{align*}
(\hat\Phi_{(k)}({\bf y})-\hat\Phi_{(k)}(\mb{x}))^2\le 1/(\mu_{\min}^2 r).
\end{align*}
Therefore, the product of the rate of transitions (which can
change $\hat\Phi_{(k)}$) and the value of $(\hat\Phi_{(k)}({\bf y})-\hat\Phi_{(k)}(\mb{x}))^2$
is upper bounded by $c_4/\mu_{\min}^2$.
Assembling all the above estimates, we obtain
\begin{align*}
A \exp(\theta \hat\Phi_{(k)}^2(\mb{x}))
& \le 2\theta\hat\Phi_{(k)}(\mb{x})
[A \hat\Phi_{(k)} (\mb{x})]
\exp(\theta \hat\Phi_{(k)}^2(\mb{x})) \\
& +\theta c_4\mu_{\min}^{-2}
\exp(\theta \hat\Phi^-(\mb{x})+\theta)
\exp(\theta \hat\Phi_{(k)}^2(\mb{x}))\\
& +2\theta^2
(|\hat\Phi_{(k)}(\mb{x})| + \mu_{\min}^{-1}r^{-1/2})^2
\exp(\theta \hat\Phi^-(\mb{x})+\theta)
\exp(\theta \hat\Phi_{(k)}^2(\mb{x}))\\
&\le  2\theta\hat\Phi_{(k)}(\mb{x})
[A \hat\Phi_{(k)} (\mb{x})]
\exp(\theta \hat\Phi_{(k)}^2(\mb{x})) \\
& +\theta c_5
\exp(\theta \hat\Phi^-(\mb{x}))
\exp(\theta \hat\Phi_{(k)}^2(\mb{x}))\\
& +c_6\theta^2
(|\hat\Phi_{(k)}(\mb{x})| + \mu_{\min}^{-1}r^{-1/2})^2
\exp(\theta \hat\Phi^-(\mb{x}))
\exp(\theta \hat\Phi_{(k)}^2(\mb{x}))
\end{align*}
where we have incorporated $e^\theta$ into constants $c_5$ and $c_6$. This is possible by our assumption $\bar\theta\le 1$.
Furthermore, for large enough $r$ we have have $ \mu_{\min}^{-1}r^{-1/2}<1$.
Thus $(|\hat\Phi_{(k)}(\mb{x})| + \mu_{\min}^{-1}r^{-1/2})^2\le 2\hat\Phi_{(k)}^2(\mb{x})+2$,
and for the appropriate choice of constants $c_7$ and $c_8$, we obtain
\begin{align*}
A \exp(\theta \hat\Phi_{(k)}^2(\mb{x}))
& \le 2\theta\hat\Phi_{(k)}(\mb{x})
[A \hat\Phi_{(k)} (\mb{x})]
\exp(\theta \hat\Phi_{(k)}^2(\mb{x})) \\
& +\theta c_7
\exp(\theta \hat\Phi^-(\mb{x}))
\exp(\theta \hat\Phi_{(k)}^2(\mb{x}))\\
& +c_8\theta^2
\hat\Phi_{(k)}^2(\mb{x})
\exp(\theta \hat\Phi^-(\mb{x}))
\exp(\theta \hat\Phi_{(k)}^2(\mb{x}))
\end{align*}


Now we will use appropriate bounds on $A \hat\Phi_{(k)} (\mb{x})$.
Observe that if the state $\mb{x}$ is such that $\hat\Phi(\mb{x})>k$ then the left hand side is at most zero.
If  $\mb{x}$ is such that $0\le \hat\Phi(\mb{x})\le k$,
then we have $A \hat\Phi_{(k)} (\mb{x})\le \bar A \hat\Phi (\mb{x})$
and can use
upper bound (\ref{eq:BoundGenerator}) on $\bar A \hat\Phi (\mb{x})$. 
Also, for all large enough $r$,
$\hat\Phi(\mb{x})<0$ implies
$A \hat\Phi_{(k)} (\mb{x}) = \bar A \hat\Phi (\mb{x})$
and we can use lower bound (\ref{eq:BoundGeneratorLower}) 
on $\bar A \hat\Phi (\mb{x})$.

Therefore, we can write
\begin{align*}
A \exp(\theta \hat\Phi_{(k)}^2(\mb{x}))
&\le -2\theta \hat\Phi(\mb{x}) I\lbrace\hat\Phi(\mb{x})\le k\rbrace\exp(\theta \hat\Phi^2(\mb{x}))\min\lbrace\hat Z(\mb{x}),a\rbrace\\
&-2c_1\theta \hat\Phi(\mb{x})\hat\Phi^+(\mb{x}) I\lbrace 0 \le \hat\Phi(\mb{x})\le k\rbrace\exp(\theta \hat\Phi^2(\mb{x}))\\
&+2c_2\theta \hat\Phi(\mb{x}) I\lbrace 0 \le \hat\Phi(\mb{x})\le k\rbrace\exp(\theta \hat\Phi^2(\mb{x}))\sum_i\hat Z_i^-\\
&+2c_3\theta \hat\Phi(\mb{x}) I\lbrace 0 \le \hat\Phi(\mb{x})\le k\rbrace\exp(\theta \hat\Phi^2(\mb{x})) \\
&-2c_{25}\theta \hat\Phi(\mb{x}) I\lbrace \hat\Phi(\mb{x}) < 0 \rbrace\exp(\theta \hat\Phi^2(\mb{x}))\sum_i\hat Z_i^+\\
&+c_7\theta I\lbrace\hat\Phi(\mb{x})\le k\rbrace\exp(\theta \hat\Phi^-(\mb{x}))\exp(\theta \hat\Phi^2(\mb{x}))\\
&+c_8\theta^2 I\lbrace\hat\Phi(\mb{x})\le k\rbrace\hat\Phi^2(\mb{x})\exp(\theta \hat\Phi^-(\mb{x}))
\exp(\theta \hat\Phi^2(\mb{x})),\\
\end{align*}
where we switched from $\hat\Phi_{(k)}(\mb{x})$ to $\hat\Phi(\mb{x})$ since 
they are equal when  $\hat\Phi\le k$.
We take the expectation of both sides with respect to the
stationary distribution of the process. Since the expectation
of the LHS is zero, and we can get rid of positive factor $\theta$,
we obtain
\begin{align}
2 \E\left[\hat\Phi I\lbrace\hat\Phi\le k\rbrace\exp(\theta \hat\Phi^2)\min\lbrace\hat Z,a\rbrace\right]
&+2c_1 \E\left[\hat\Phi\hat\Phi^+ I\lbrace 0 \le \hat\Phi\le k\rbrace\exp(\theta \hat\Phi^2)\right]\notag\\
&\le 2c_2 \E\left[\hat\Phi^+ I\lbrace 0 \le \hat\Phi\le k\rbrace\exp(\theta (\hat\Phi^+)^2)\sum_i\hat Z_i^-\right] \label{eq:MainInequality}\\
&+2c_3 \E\left[\hat\Phi^+ I\lbrace 0 \le \hat\Phi\le k\rbrace\exp(\theta (\hat\Phi^+)^2)\right] \notag\\
&+ 2c_{25} \E\left[\hat\Phi^- I\lbrace \hat\Phi < 0 \rbrace\exp(\theta (\hat\Phi^+)^2)\sum_i\hat Z_i^+\right] \notag\\
&+c_7\theta \E\left[I\lbrace\hat\Phi\le k\rbrace\exp(\theta \hat\Phi^-)\exp(\theta \hat\Phi^2)\right]\notag\\
&+c_8\theta \E\left[I\lbrace\hat\Phi\le k\rbrace\hat\Phi^2\exp(\theta \hat\Phi^-)\exp(\theta \hat\Phi^2)\right],\notag
\end{align}
where in the first two terms after $\le$, we replaced $\hat\Phi$ by $\hat\Phi^+$ 
because they are equal when $\hat\Phi \ge 0$;
similarly, in the third term after $\le$, we replaced
$-\hat\Phi$ and $\hat\Phi^2$ by $\hat\Phi^-$ and $(\hat\Phi^-)^2$,
respectively,
using condition $\hat\Phi < 0$.

Let us analyze term by term and obtain some bounds that are uniform in $r$ and $k$.
First consider
\begin{align*}
2 \E\left[\hat\Phi I\lbrace\hat\Phi\le k\rbrace\exp(\theta \hat\Phi^2)\min\lbrace\hat Z,a\rbrace\right]
\end{align*}
When $\hat\Phi$ and $\hat Z$ have the same sign, the term is non-negative and we can drop it from the inequality.
When $\hat\Phi\le 0$ and $\hat Z \ge 0$ we have
\begin{align*}
|2 \hat\Phi I\lbrace\hat\Phi\le k\rbrace\exp(\theta \hat\Phi^2)\min\lbrace\hat Z,a\rbrace|\le
2 a \hat\Phi^-\exp(\theta \hat\Phi^2)\le 2 a \mu_{\max}(\sum_i \hat Z_i^-)
\exp(\theta \mu_{min}^{-2}(\sum_i \hat Z_i^-)^2).
\end{align*}

For sufficiently small $\theta>0$,
applying Lemma~\ref{lem-bound-on-square}, we have
\begin{align*}
\limsup_{r\rightarrow\infty}2
\E\left[I\lbrace\hat\Phi\le 0,~\hat Z \ge 0 \rbrace
\Big|\hat\Phi I\lbrace\hat\Phi\le k\rbrace\exp(\theta \hat\Phi^2)\min\lbrace\hat Z,a\rbrace\Big|\right]
\le c_9,
\end{align*}
for some $0<c_9<\infty$ (independent of $k$).
Finally, when $\hat\Phi \ge 0,\hat Z\le 0$ we have
\begin{align*}
2 |\hat\Phi I\lbrace\hat\Phi\le k\rbrace\exp(\theta \hat\Phi^2)\min\lbrace\hat Z,a\rbrace|\le
2 \hat\Phi^+ I\lbrace\hat\Phi\le k\rbrace\exp(\theta (\hat\Phi^+)^2)\sum_i\hat Z_i^-.
\end{align*}
Observe that we already have the expectation of
such a term on the right hand side of our inequality (\ref{eq:MainInequality}), up to a constant factor.

For the second term in the inequality (\ref{eq:MainInequality}), we have
\begin{align}
\label{eq:term2}
2c_1 \E\left[\hat\Phi\hat\Phi^+  I\lbrace\hat\Phi\le k\rbrace\exp(\theta \hat\Phi^2)\right]=
2c_1 \E\left[(\Phi^+)^2  I\lbrace\hat\Phi\le k\rbrace\exp(\theta (\hat\Phi^+)^2)\right]
= g,
\end{align}
where $g=g(r,k,\theta)$ is simply a notation for the RHS, for the ease of reference.

We leave the first term after inequality (\ref{eq:MainInequality}) untouched for now.
The second term after inequality (\ref{eq:MainInequality}) is
upper bounded by $g/4+c_{10}$ for some constant $c_{10}>0$,
by considering the cases $\hat\Phi<c$ and $\hat\Phi\ge c$ separately,
for a sufficiently large fixed $c>0$.
For the third term after inequality (\ref{eq:MainInequality}),
we use the fact that for some fixed $c_{26}>0$,
condition $\hat\Phi < 0$ implies $\sum_i\hat Z_i^+ \le c_{26} \sum_i\hat Z_i^-$,
and therefore, by Lemma~\ref{lem-bound-on-square}, 
this term is uniformly bounded for all $r$ and all sufficiently small $\theta>0$.

Next consider
\begin{align*}
c_7 \E\left[I\lbrace\hat\Phi\le k\rbrace\exp(\theta \hat\Phi^-)\exp(\theta \hat\Phi^2)\right].
\end{align*}
The expectation on the event
$\{\hat\Phi\le 0\}$ is bounded by
$c_7\E[\exp(\theta\hat\Phi^-)\exp(\theta(\hat\Phi^-)^2))]$ which, by Lemma~\ref{lem-bound-on-square}, is bounded by
 some $c_{11}>0$ and sufficiently small $\theta>0$, uniformly in $r$.
The expectation on the event
$\{\hat\Phi> 0\}$  is bounded by
$c_7\E\left[I\lbrace\hat\Phi\le k\rbrace\exp(\theta (\hat\Phi^+)^2)\right]$
which is
upper bounded by $g/4+c_{12}$ for some constant $c_{12}>0$, again
by considering the cases $\hat\Phi<c$ and $\hat\Phi\ge c$ separately,
for a sufficiently large fixed $c>0$.

Finally, for the last term in (\ref{eq:MainInequality})
\begin{align*}
c_8\theta \E\left[I\lbrace\hat\Phi\le k\rbrace\hat\Phi^2\exp(\theta \hat\Phi^-)\exp(\theta \hat\Phi^2)\right],
\end{align*}
the expectation on the
event $\{\hat\Phi<0\}$ is uniformly bounded by some $c_{13}\theta$,
for sufficiently small $\theta>0$,
by  applying Lemma~\ref{lem-bound-on-square}.
The expectation on the
event $\{\hat\Phi\ge 0\}$
is bounded by
\begin{align*}
c_8\theta \E\left[I\lbrace\hat\Phi^+\le k\rbrace(\hat\Phi^+)^2\exp(\theta (\hat\Phi^+)^2)\right]
\end{align*}
For sufficiently small $\theta>0$,
to ensure $c_8\theta<(1/2)c_1$, this term is upper bounded by $g/4$.
Thus we can rewrite (\ref{eq:MainInequality}) as
\begin{align}
\label{eq:MainInequality1}
c_{21} \E\left[(\Phi^+)^2  I\lbrace\hat\Phi\le k\rbrace\exp(\theta (\hat\Phi^+)^2)\right]
&\le c_{22} \E\left[\hat\Phi^+ I\lbrace\hat\Phi\le k\rbrace\exp(\theta (\hat\Phi^+)^2)\sum_i\hat Z_i^-\right]+c_{14},
\end{align}
for some positive $c_{21},c_{22},c_{14}$.
We analyze the term on the right-hand side. Fix $C>0$. Then
\begin{align*}
\hat\Phi^+ I\lbrace\hat\Phi\le k\rbrace\exp(\theta (\hat\Phi^+)^2)\sum_i\hat Z_i^-
&=\hat\Phi^+ I\lbrace C\sum_i\hat Z_i^-\le \hat\Phi\le  k\rbrace\exp(\theta (\hat\Phi^+)^2)\sum_i\hat Z_i^-\\
&+\hat\Phi^+ I\lbrace \hat\Phi <  \min(k,C\sum_i\hat Z_i^-)\rbrace\exp(\theta \hat\Phi^2)\sum_i\hat Z_i^-\\
&\le (1/C)(\hat\Phi^+)^2 I\lbrace\hat\Phi\le  k\rbrace\exp(\theta (\hat\Phi^+)^2)\\
&+C(\sum_i\hat Z_i^-)^2\exp(\theta C^2(\sum_i\hat Z_i^-)^2)\\
\end{align*}
The second summand is bounded in expectation (by Lemma~\ref{lem-bound-on-square})
for sufficiently small $\theta>0$.
Consider
 constant $C$ such that $c_{21}-c_{22}/C=c_{21}/2$ (that is $C=2 c_{22}/c_{21}$~).
By the choice of $C$,
if we apply the above estimates to the RHS of \eqn{eq:MainInequality1},
we obtain
\begin{align*}
(c_{21}/2)&\E\left[(\hat\Phi^+)^2 I\lbrace\hat\Phi\le  k\rbrace\exp(\theta (\hat\Phi^+)^2)\right]
\le c_{24}
\end{align*}
for some $c_{24}>0$.
We have identified a constant $c_{15}$ such that for every $k$
\begin{align*}
\limsup_{r\rightarrow\infty}\E\left[(\hat\Phi^+)^2 I\lbrace\hat\Phi\le  k\rbrace\exp(\theta (\hat\Phi^+)^2)\right]\le c_{15}<\infty.
\end{align*}
Since, again by Lemma~\ref{lem-bound-on-square},
for sufficiently small $\theta>0$,  we also have for some $c_{16}$
\begin{align*}
\limsup_{r\rightarrow\infty}\E\left[(\hat\Phi^-)^2 I\lbrace\hat\Phi\le  k\rbrace\exp(\theta (\hat\Phi^-)^2)\right]\le c_{16}<\infty,
\end{align*}
we obtain for $c_0\triangleq c_{15}+c_{16}$
\begin{align*}
\limsup_{r\rightarrow\infty}\E\left[\hat\Phi^2 I\lbrace\hat\Phi\le  k\rbrace\exp(\theta (\hat\Phi)^2)\right]\le c_0<\infty.
\end{align*}
This implies (\ref{eq-abandon})
and the proof of  Proposition~\ref{prop:abandon} is complete.
\end{proof}

We now complete the proof of Theorem~\ref{th-main-abandonments}(ii).
To show (\ref{eq-est333-for-plus}), we use the same argument
as the one for Corollary~\ref{coro:WeakLimitTight},
with $(\hat Z^+(\infty))^2$ and $(\hat Z^{r,+}(\infty))^2$ replacing
$|\hat Z(\infty)|$ and $|\hat Z^r(\infty)|$, respectively.

\bibliographystyle{amsalpha}
\bibliography{bibliography}

\end{document}